\documentclass[12pt,a4paper]{amsart}
\usepackage{amssymb}

\calclayout

\renewcommand{\:}{\colon}
\renewcommand{\.}{\mskip.5\thinmuskip}

\newcommand{\ot}{\otimes}
\newcommand{\rarrow}{\longrightarrow}

\newcommand{\lrarrow}{\.\relbar\joinrel\relbar\joinrel\rightarrow\.}

\DeclareMathOperator{\Fr}{Fr}
\DeclareMathOperator{\GL}{GL}
\DeclareMathOperator{\PGL}{PGL}
\DeclareMathOperator{\Gal}{Gal}
\DeclareMathOperator{\Spec}{Spec}
\DeclareMathOperator{\Hom}{Hom}
\DeclareMathOperator{\Aut}{Aut}
\DeclareMathOperator{\chr}{char}

\newcommand{\A}{\mathbb A}
\newcommand{\Z}{\mathbb Z}
\newcommand{\Q}{\mathbb Q}
\newcommand{\C}{\mathbb C}
\renewcommand{\P}{{\mathbb P}}

\newcommand{\Sch}{\mathsf{Sch}}
\newcommand{\GrOut}{\mathsf{GrOut}}
\newcommand{\Grd}{\mathsf{Grd}}
\newcommand{\Fun}{\mathsf{Fun}}
\newcommand{\pf}{{\mathsf{pf}}}
\newcommand{\cn}{{\mathsf{cn}}}
\newcommand{\ct}{{\mathsf{ct}}}

\theoremstyle{plain}
\newtheorem{thm}{Theorem}
\newtheorem*{lem}{Lemma}

\begin{document}

\title{Noncommutative Kummer theory}

\author{Leonid Positselski}

\address{Department of Mathematics, Faculty of Natural Sciences,
University of Haifa, Mount Carmel, Haifa 31905, Israel}
\email{posic@mccme.ru}

\begin{abstract}
 We discuss lifting properties of continuous homomorphisms from
absolute Galois groups into (pro)finite groups.
 An analogy with the Langlands program is pointed out in the beginning
of the note.
\end{abstract}

\maketitle

 The Langlands program is a noncommutative generalization of class
field theory.
 Class field theory provides a description of abelian extensions of
number fields.

 Another description of abelian extensions of fields containing
enough roots of unity is provided by the (much more elementary)
Kummer theory.
 The two descriptions are related by the Tate duality in cohomology.
Unlike the class field theory, the Kummer theory applies in
a uniform manner to local and global fields, and also to fields
of arbitrary transcendence degree.

 Absolute Galois groups of arbitrary fields are very special objects
among profinite groups.
 Some of their special properties are easy to prove; some are very
difficult theorems, some others are conjectures, and probably there
are still more such properties yet to be discovered.
 The properties that are already known, either definitely or
conjecturally, do not fit into a single coherent picture.
 It appears that whatever angle one chooses to look on the absolute
Galois groups, one is likely to find that, looking from this angle,
they are very special in a separate way, seemingly unrelated to
their other special properties (see, e.~g.,
\cite{Lur,Pdivis,Pbogom,MO}).

 In particular, the cohomology of absolute Galois groups with
cyclotomic coefficients are described by the Milnor--Bloch--Kato
conjecture, proven by Rost, Voevodsky, \emph{et al.}~\cite{Voev}.
 In the case of $H^1$, this reduces to the Kummer theory.
 What can one say about the first cohomology of absolute Galois
groups with noncommutative (constant or close-to-constant)
coefficients?

 Attempts to answer this question bear some vague resemblance to
the Langlands story in that representations of a finite Galois
group $G=\Gal(L/K)$ in the linear groups like $\GL(n,K)$ or
$\PGL(2,K)$ get involved (if only because of the importance
of Hilbert theorem~90 for $\GL$).

\bigskip

 One of the formulations of the Kummer theory for fields with roots
of unity is the lifting property for homomorphisms from absolute
Galois groups $G_K=\Gal(\overline{K}/K)$ into cyclic groups.
 More precisely, for any prime number~$l$ and any field $K$ containing
$\!\sqrt[l^\infty]{1}$, any continuous group homomorphism
$f\:G_K\rarrow\Z/l^n$ can be lifted to a continuous homomorphism
$\tilde f\:G_K\rarrow\Z_l$.
 The explanation is that the homomorphism~$f$ corresponds to a field
extension $K[\!\sqrt[l^n]{b}]/K$ for some element $b\in K$, and
the extension $K[\!\sqrt[l^\infty]{b}]/K$ then provides
the desired lifting.
 Hence one way to approach noncommutative Kummer theory would be to
look into lifting properties of homomorphisms from absolute Galois
groups into noncommutative (pro)finite groups.

 Here is one example of a noncommutative lifting property of
the absolute Galois groups of arbitrary fields over a given field.

\begin{thm}
 Let $K$ be a field containing\/~$\overline{\Q}$.
 Then any profinite group homomorphism $G_K\rarrow S_n$ can be lifted
to a profinite group homomorphism $G_K\rarrow\widehat{\mathrm{Br}}_n$
(where $S_n$ denotes the symmetric group and\/ $\widehat{\mathrm{Br}}_n$
is the profinite completion of the braid group\/~$\mathrm{Br}_n$
mapping onto~$S_n$).
\end{thm}

 More generally, let $G$ be a finite group and $k$~be a fixed field.
 Let us say that $c\:\widetilde{G}\rarrow G$ is a \emph{weakly universal
absolute Galois covering of $G$ over~$k$} if $\widetilde{G}$ is
the absolute Galois group of a field over~$k$, \ $c$~is a profinite
group homomorphism, and for any field $K$ over~$k$, any profinite
group homomorphism $f\:G_K\rarrow G$ can be lifted to a homomorphism
$\tilde f\:G_K\rarrow \widetilde G$ so that $f=c\tilde f$.

\begin{thm}
 Let $k$ be a field and $V$ a faithful linear representation of
a finite group $G$ in a finite-dimensional $k$\+vector space.
 Denote by $k(V)$ the field of rational functions on $V$ and
by $E=k(V)^G$ its subfield of $G$\+invariant elements.
 Then the natural homomorphism $G_E=\Gal(\overline{k(V)}/k(V)^G)
\rarrow G=\Gal(k(V)/k(V)^G)$ is a weakly universal absolute
Galois covering of $G$ over~$k$.
\end{thm}

 In other words, any continuous group homomorphism $f\:G_K\rarrow G$
with a field $K$ containing $k$ lifts to a continuous homomorphism
$\tilde f\:G_K\rarrow G_E$.

 Here is a more concrete-looking result (cf.~\cite{Ser}).

\begin{thm}
 Let $k$ be an algebraically closed field, $G$ a finite subgroup
in\/ $\PGL(2,k)$, and $G'$ a finite subgroup in\/ $\GL(2,k)$ mapped
onto $G$ by the natural projection.
 Let\/ $\Fr\rarrow G$ be a homomorphism onto $G$ from a free profinite
group\/~$\Fr$.
 Then for any field $K$ over $k$, any profinite group homomorphism
$G_K\rarrow G$ that factors through the epimorphism $G'\rarrow G$
also factors through the epimorphism\/ $\Fr\rarrow G$.
\end{thm}

 For example, let $D_n$ denote the dihedral group of order~$2n$.
 Then a continuous homomorphism $G_K\rarrow D_n$ lifts to a free
profinite group covering $D_n$ whenever it lifts to $D_{2n}$.
 The latter condition is necessary here: for any even~$n$, there
exists a field $K$ over $\overline{\Q}$ and a homomorphism
$G_K\rarrow D_n$ that cannot be lifted to $D_{2n}$.
 The situation is similar for the groups $G=A_4$, $S_4$, and~$A_5$.

 E.~g., a homomorphism $G_K\rarrow \Z/2\times\Z/2$ lifts to
the free profinite group with two generators $\Fr_2$ if and only
if the cup-product of the two related classes in $H^1(G_K,\Z/2)$
vanishes in $H^2(G_K,\Z/2)$.

\begin{proof}[Sketch of proof of Theorem~1]
 Let $X\subset\A^n_{\overline\Q}$ denote the $\overline\Q$\+variety of
all monic polynomials $q(y)=y^n+q_{n-1}y^{n-1}+\dotsb+q_0$ without
multiple roots and $Y\subset\A^n_{\overline\Q}$ the complement to
the union of the hyperplanes $\{y_i=y_j\}$, \ $1\le i<j\le n$.
 Let $\sigma\: Y\rarrow X$ be the natural Galois covering with
the Galois group~$S_n$.
 Notice that profinite group homomorphisms $f\:G_K\rarrow S_n$
come from polynomials of degree~$n$ over $K$ having no multiple
roots.
 In other words, any such homomorphism~$f$, viewed up to a conjugation
in $S_n$, corresponds to a pullback of the covering~$\sigma$ via some
$K$\+point $q\:\Spec K\rarrow X$ on $X$.
 The choice of a point~$q$ thus provides a lifting of~$f$ to
a homomorphism of \'etale fundamental groups $\tilde f\:G_K\rarrow
\pi_1^{\text{\'et}}(X)\simeq\widehat{\pi_1(X(\C))}\simeq
\widehat{\mathrm{Br}}_n$, defined up to a conjugation in
$\widehat{\mathrm{Br}}_n$.
\end{proof}

 Quite generally, connected groupoids and isomorphism classes of
functors between them form a category equivalent to the category
$\GrOut$ of groups and ``outer homomorphisms'' (homomorphisms up to
a conjugation in the target).
 Similarly, one can consider groupoids in which the automorphism
groups of objects are profinite rather than discrete groups (while
objects themselves still form a discrete set).
 The \'etale fundamental group is a functor from the category of
connected schemes to the category of profinite groups and
outer homomorphisms
$$
 \pi_1^{\text{\'et}}\:\Sch_\cn\lrarrow\GrOut_\pf
$$
which extends naturally to the ``\'etale fundamental groupoid''
functor $\Sch\rarrow\Grd_\pf$ from the category of schemes to
the category of profinite groupoids.

 Given two groupoids $\Gamma$ and $\Delta$, the category
$\Fun(\Gamma,\Delta)$ of functors from $\Gamma$ to $\Delta$ is
again a groupoid, whose set of connected components is
the set of morphisms $\Hom_{\Grd}(\Gamma,\Delta)$ in
the category~$\Grd$.
 Similarly, when $\Gamma$ is a profinite groupoid and $\Delta$
is a groupoid with finite automorphism groups of objects,
the category of continuous functors $\Fun_\ct(\Gamma,\Delta)$
is a groupoid with finite automorphism groups, whose set of
connected components is $\Hom_{\Grd_\pf}(\Gamma,\Delta)$.
 Given a connected scheme $X$ and a finite group $G$, the groupoid
$\Fun_\ct(\pi_1^{\text{\'et}}(X),G)$ is equivalent to the category of
schemes $Y$ endowed with a free action of $G$ and an isomorphism
$Y/G\simeq X$.

 Now let $Y$ be a scheme over a field $k$ endowed with a free action
of a finite group $G$, and let $X=Y/G$ be the quotient.
 Then there is a groupoid in the category of schemes over~$k$ with
the scheme of objects $X$ and the scheme of morphisms $(Y\times_kY)/G$.
 The source and target morphisms $(Y\times_kY)/G\rightrightarrows X$
are the projections onto the first and second factor, the unit
morphism $X\rarrow(Y\times_kY)/G$ comes from the diagonal
$Y\rarrow Y\times_kY$, the morphism of the inverse element
switches the factors, and the composition
$(Y\times_kY)/G\times_X(Y\times_kY)/G\simeq(Y\times_kY\times_kY)/G
\rarrow(Y\times_kY)/G$ drops the middle factor.
 Let us denote this groupoid in $\Sch/k$ by $\Gamma_G(Y)=
(X,\,(Y\times_kY)/G)$.

 Given a field $K$ over $k$, the functor of $K$\+points
$\Hom_{\Sch/k}(\Spec K,{-})$ transforms the groupoid $\Gamma_G(Y)$
into a groupoid $\Gamma_G(Y)(K)$ in the category of (discrete) sets
with the set of objects $X(K)$ and the set of
morphisms $((Y\times_kY)/G)(K)$.
 The rule assigning to a $K$\+point $q\:\Spec K\rarrow X$
the scheme $M(q)=\Spec K\times_X Y$ with the induced free action of
$G$ and the natural identification $M(q)/G\simeq\Spec K$
extends naturally to a \emph{fully faithful} functor of groupoids
$$
 M=M_{G,Y}(K)\:\Gamma_G(Y)(K)\lrarrow\Fun_\ct(G_K,G).
$$
 Indeed, given two $K$\+points $q_1,q_2\:\Spec K\rightrightarrows X$,
(iso)morphisms of schemes $M(q_1)\rarrow M(q_2)$ over $\Spec K$
correspond bijectively to $K$\+points of the scheme
$\Spec K\times_{X\times X}(Y\times_kY)/G\simeq
(M(q_1)\times_{\Spec K}M(q_2))/G$ (because given two principal
$G$\+sets $M_1$ and $M_2$, there is a natural bijection between
$G$\+isomorphisms $M_1\rarrow M_2$ and elements of the set
$(M_1\times M_2)/G$).
 Thus the groupoid $\Gamma_G(Y)$ in $\Sch/k$ can be thought of
as the \emph{classifying groupoid} for those outer homomorphisms
$G_K\rarrow G$ that come from $K$\+points on~$X$.

\bigskip

 Given a $k$\+variety $Y$ with a free action of finite group $G$, when
does an outer homomorphism $f\:G_K\rarrow G$ come from a $K$\+point
on $X=Y/G$?
 The following criterion can be found in~\cite{Ser}.
 The action of $G$ in $Y_K=\Spec K\times_{\Spec k}Y$ provides
an embedding of $G$ into the automorphism group $\Aut(Y_K)
\subset\Aut(Y_{\overline K})$, hence the homomorphism~$f$ defines
a cocycle representing a Galois cohomology class in
$H^1(G_K,\Aut(Y_{\overline K}))$.
 Twisting the $K$\+structure on $Y_K$ by this cocycle, one obtains
a $K$\+variety $Y_f$ together with an isomorphism of
$\overline K$\+varieties $Y_{\overline K}\simeq (Y_f)_{\overline K}$.
 So there is a natural bijection between the sets of
$\overline K$\+points $Y(\overline K)\simeq Y_f(\overline K)$.
 The homomorphism~$f$ comes from a point $x\in X(K)$ if and only if
there is a $K$\+point $y_f\in Y_f(K)$ such that the corresponding
point $y\in Y(\overline K)$ lies over~$x$.
 Thus the groupoid $\Gamma_G(Y)$ classifies those homomorphisms
$f\:G_K\rarrow G$ for which the variety $Y_f$ has a $K$\+point.

 Now let $A$ be a $k$\+algebraic group acting in $Y$ and $G\subset A(k)$
be a finite subgroup whose action in $Y$ is free.
 Using the embedding $G\rarrow A(k)\subset A(K)\subset A(\overline K)$,
a homomorphism $f\:G_K\rarrow G$ defines a Galois cohomology
class $\xi_f\in H^1(G_K,A(\overline K))$.
 Suppose that $\xi_f=1$; then the cocycle corresponding to~$f$
is the coboundary of an element $a_f\in A(\overline K)$.
 Assume further that there is a $K$\+point $y_0\in Y(K)$.
 Then the point $a_f(y_0)\in Y(\overline K)$ corresponds to
a $K$\+point on~$Y_f$.
 Hence the outer homomorphism~$f$ comes from a $K$\+point on $X$
in this case.

 Taking $Y=A$ (with the action of $A$ in itself by right
multiplications), one discovers that this sufficient condition is
also necessary in this situation, so the groupoid $\Gamma_G(A)$
classifies those homomorphisms $f\:G_K\rarrow G$ for which $\xi_f=1$.
 Similarly, if $Y$ is the principal homogeneous $A$\+space
corresponding to a Galois cohomology class
$\xi_0\in H^1(G_k,A(\overline k))$, then the groupoid $\Gamma_G(Y)$
classifies the homomorphisms~$f$ for which $\xi_f$ is equal to
the image of $\xi_0$ in $H^1(G_K,A(\overline K))$.

\begin{proof}[Sketch of proof of Theorem~2]
 This result is closely related to the ``GAL\+realization'' technique
from inverse Galois theory~\cite{Vol}.
 Let $v_1$,~\dots, $v_n$ be a basis in $V$ and let $L=K(z_1,\dotsc,z_n)$
be the field of rational functions in $n$~transcendental variables
over~$K$.
 Then the natural surjective homomorphism $G_L\rarrow G_K$ splits
into a semidirect product (it suffices to pick, e.~g., a full flag
of smooth absolutely irreducible subvarieties in $\A^n_K$ in order
to obtain a splitting).
 Hence it suffices to show that the composition $f_L\:G_L\rarrow G_K
\rarrow G$ can be lifted to a homomorphism $G_L\rarrow G_E$.

 Since $H^1(G_K,\GL(n,\overline K))=1$, the cocycle obtained by
composing the homomorphism $f\:G_K\rarrow G$ with the embedding
$G\rarrow\GL(V)\rarrow\GL(K\ot_kV)\rarrow\GL(\overline K\ot_k V)$
is the coboundary of a certain invertible matrix $a_f\in
\GL(\overline K\ot_kV)$.
 Consider the scheme $Y=\Spec L\ot_k k(V)$ over~$L$; it is
the intersection of open complements in the affine space $\A^n_L$ to
the zero sets of all polynomial equations with coefficients in~$k$.
 The vector $y_0=z_1v_1+\dotsb+z_nv_n\in L\ot_kV$ defines
an $L$\+point in $Y$, and the vector $a_f(y_0)\in\overline L\ot_kV$
is an $\overline L$\+point in~$Y$.
 Taking the image of $a_f(y_0)$ under the projection map
$Y\rarrow X=Y/G$ (where $G$ acts in $Y$ via $k(V)$), we obtain
an $L$\+point $q\:\Spec L\rarrow X$ in $X$ such that
the homomorphism~$f_L$ corresponds to the pulback of the covering
$Y\rarrow X$ along the morphism~$q$.

 It remains to notice that the morphism $\Spec L\ot_k\overline{k(V)}
\rarrow\Spec L\ot_k k(V)^G=X$ is an infinite Galois covering with
the Galois group isomorphic to~$G_E$.
 Pulling it back along the point~$q$ provides the desired lifting
$\tilde f_L\:G_L\rarrow G_E$.
\end{proof}

\begin{proof}[Sketch of proof of Theorem~3]
 Suppose that a profinite group homomorphism $f\:G_K\allowbreak
\rarrow G\subset\PGL(2,k)$ can be lifted to $G'\subset\GL(2,k)$.
 Then the cocycle obtained by composing~$f$ with the embeddings
$\PGL(2,k)\subset\PGL(2,K)\subset\PGL(2,\overline K)$ lifts to
a cocycle with coefficients in $\GL(2,\overline K)$.
 Hence this cocycle is the coboundary of some group element
$a_f\in\PGL(2,\overline K)$.

 Arguing as in the proof of Theorem~2, consider the field of
rational functions in a transcendental variable $L=K(z)$.
 We have to show that the composition $f_L\:G_L\rarrow G_K
\rarrow G$ can be lifted to a profinite group homomorphism
$\tilde f_L\:G_L\rarrow\Fr$.

 Let the group $\PGL(2)$ act in the projective line $\P^1$.
 Consider the scheme $Y=\Spec L\ot_k k(\P^1)$ over $L$;
it is the projective line $\P^1_L$ with all the points defined
over~$k$ thrown out.
 Let $y_0\in Y(L)$ be a point with the coordinate transcendental
over~$K$; then $a_f(y_0)$ in an $\overline L$\+point in~$Y$.
 The finite group $G$ acts freely in $Y$; and the image of
$a_f(y_0)$ under the projection map $Y\rarrow X=Y/G$ is
an $L$\+point $q\:\Spec L\rarrow X$ to which the homomorphism~$f_L$
corresponds.

 The natural morphism $\Spec L\ot_k\overline{k(\P^1)}\rarrow
\Spec L\ot_kk(\P^1)^G$ is an infinite Galois covering with
the Galois group isomorphic to $G_E=\Gal(\overline{k(\P^1)}/k(\P^1)^G)$.
 Pulling back this covering along the point~$q$ provides a lifting
of the homomorphism~$f$ into the group~$G_E$.
 Now by Tsen's theorem the field $E=k(\P^1)^G$ has cohomological
dimension~$1$, so any epimorphism onto $G_E$ from a profinite
group splits.
\end{proof}

 It remains to explain why (and when) the condition of liftability
to $\GL(2)$ is nontrivial in Theorem~3.
 The following lemma is due to Bogomolov.

\begin{lem}
 For any finite group $G$ and an algebraically closed field~$k$ with\/
$\chr k\ne l$, there exists a field $E$ over $k$ and a profinite
group homomorphism $G_E\rarrow G$ such that the induced map
$H^2(G,\Q_l/\Z_l)\rarrow H^2(G_E,\Q_l/\Z_l)$ is injective.
\end{lem}

\begin{proof}
 Let $F/E$ be a field extension with the Galois group $\Gal(F/E)=G$,
so $G_E/G_F\simeq G$.
 Then from the inflation-restriction sequence one concludes that,
for any abelian group $A$ endowed with the trivial action of $G_E$,
the map $H^2(G,A)\rarrow H^2(G_E,A)$ is injective if and only if
the map $H^1(G_E,A)\rarrow H^1(G_F,A)^G$ is surjective.
 Taking $A=\Q_l/\Z_l$ and identifying this group with the group
of $l^\infty$\+roots of unity $\mu_{l^\infty}\subset k=\overline k
\subset E$, we have $H^1(G_E,\mu_{l^\infty})=E^*\ot_\Z\Q_l/\Z_l$
and $H^1(G_F,\mu_{l^\infty})=F^*\ot_\Z\Q_l/\Z_l$.
 So it suffices to find a field extension $F/E$ over $k$ with
$\Gal(F/E)=G$ such that the map $E^*\ot_\Z\Q/\Z\rarrow
(F^*\ot_\Z\Q/\Z)^G$ is surjective.
 From the short exact sequence $F^*/k^*\rarrow F^*/k^*\ot_\Z\Q
\rarrow F^*\ot_\Z\Q/\Z$ and the observation that the group~$k^*$
is divisible (so $(F^*/k^*)^G=E^*/k^*$), we further conclude that
it is enough to produce a field extension $F/E$ for which
$H^1(G,F^*/k^*)=0$.

 Now let $V$ be a finite-dimensional $k$\+vector space with a faithful
action of~$G$.
 Set $F=k(V)$ and $E=k(V)^G$.
 Then $F^*/k^*$ is a free abelian group with a natural $\Z$\+basis
formed by the irreducible polynomials in $\dim V$ variables over~$k$.
 This basis is permuted by the action of $G$, so $H^1(H,\Z)=0$ for
any finite group $H$ acting trivially in $\Z$ implies
$H^1(G,F^*/k^*)=0$.
\end{proof}

 For any field $E$ containing $k=\overline k$, the maps
$H^2(G_E,\Z/l^n)\rarrow H^2(G_E,\Q_l/\Z_l)$ induced by
the embeddings $\Z/l^n\rarrow\Q_l/\Z_l$ are injective.
 Hence it follows from the lemma that a cohomology class
in $H^2(G,\Z/l^n)$ dies in $H^2(G_E,\Z/l^n)$ for all fields $E$
over $k$ and all homomorphisms $G_E\rarrow G$ if and only if
it dies in $H^2(G,\Q_l/\Z_l)$, i.~e., belongs to the image
of some Bockstein map $H^1(G,\Z/l^m)\rarrow H^2(G,\Z/l^n)$.
 Thus, given any central extension of finite groups
$0\rarrow\Z/l^n\rarrow G'\rarrow G\rarrow 1$ which is not induced
from a central extension of cyclic groups $\Z/l^{n+m}\rarrow\Z/l^m$,
there exists a homomorphism $G_E\rarrow G$ that cannot be
lifted to~$G'$.

\bigskip

 Questions: (1) The Kummer theory can be formulated for fields without
roots of unity using the Galois cohomology with cyclotomic
coefficients.
 Of course, the Kummer theory for fields without roots of unity could
be also formulated as a lifting property in the spirit of Theorem~2
(as in the examples of universally solvable embedding problems
in~\cite{Lur}), but the formulation in terms of the cyclotomic
cohomology is in some sense clearly preferable.

 Is there a similar cohomological formulation of the above results
(such as Theorems~1 and~3), extending them to fields without
algebraically closed subfields using noncommutative Galois
cohomology with (perhaps close-to-trivial, but) nontrivial
coefficients rather than lifting properties?

\medskip

 (2) The above reads more like an exposition in inverse Galois theory
(perhaps because Question~(1) remains unanswered).
 Is there any actual relation with the Langlands program?

\bigskip

\textbf{Acknowlegdement.}
 This note was originally written in responce to an invitation of
the organizers of the workshop ``Geometric methods in the mod~$p$
local Langlands correspondence'', which took place in Pisa between
June~6--10, 2016, to suggest questions and topics for discussion
at the workshop.
 The author is grateful to the organizers for the stimulating
invitation to suggest a discussion topic, as well as for the invitation
to the workshop.


\medskip
\begin{thebibliography}{9}

\bibitem{Lur}
 B.~B.~Lur'e.
   On universally solvable embedding problems. (Russian)
\textit{Trudy Mat.\ Inst.\ Steklov.}\ \textbf{183}, Nauka, Leningrad,
p.~121--126, 1990.  English translation in: \textit{Proceedings of
the Steklov Institute of Mathematics} \textbf{183}, p.~141--147, 1991.

\bibitem{Pdivis}
 L.~Positselski.
   Galois cohomology of certain field extensions and the divisible
case of Milnor--Kato conjecture.
\textit{K\+Theory} \textbf{36}, \#1--2, p.~33--50, 2005.
\texttt{arXiv:math.KT/0209037}

\bibitem{Pbogom}
 L.~Positselski.
   Koszul property and Bogomolov's conjecture.
\textit{Internat.\ Math.\ Research Notices} \textbf{2005}, \#31, 
p.~1901--1936.  Postpublication arXiv version:
\texttt{arXiv:1405.0965v2 [math.KT]}

\bibitem{MO}
``Profinite groups as absolute Galois groups''.
MathOverflow question \texttt{http://mathoverflow.
net/questions/232930/profinite-groups-as-absolute-galois-groups/}

\bibitem{Ser}
 J.-P.~Serre.
   Extensions icosa\'edriques. Letter to J.~D.~Gray.
\textit{Seminaire de Th\'eorie de Nombres de Bordeaux} \textbf{9}
(1979--1980), expos\'e n$^\circ$\,19, p.~1--8.

\bibitem{Voev}
 V.~Voevodsky.
   On motivic cohomology with $\mathbf Z/l$-coefficients.
\textit{Annals of Math.}\ \textbf{174}, \#1, p.~401--438, 2011.
\texttt{arXiv:0805.4430 [math.AG]}

\bibitem{Vol}
 H.~V\"olklein.
   Groups as Galois groups: An introduction.
Cambridge Studies in Advanced Mathematics 53, Cambridge University
Press, 1996.

\end{thebibliography}
\end{document}